\documentclass[a4paper,12pt, oneside]{amsart}

\usepackage[normalem]{ulem}

\usepackage{amsthm}
\usepackage{amssymb}
\usepackage{amsfonts}
\usepackage{amsmath}
\usepackage{amsthm}
\usepackage{geometry}
\usepackage{ae,aecompl}
\usepackage{eurosym}
\usepackage{verbatim}
\usepackage{caption}
\usepackage{url}
\usepackage{hyperref}

\newcounter{stepcounter}
\newtheoremstyle{smallcaps}
    {3pt}                    
    {3pt}                    
    {\itshape}                   
    {}                           
    {\sc}                   
    {.}                          
    {.5em}                       
    {}  
    
\newtheoremstyle{smallcapsdef}
    {3pt}                    
    {3pt}                    
    {}                   
    {}                           
    {\sc}                   
    {.}                          
    {.5em}                       
    {}  

\theoremstyle{plain}

\newtheorem{thm}{Theorem}[section]

\newtheorem{lem}[thm]{Lemma}
\newtheorem{prop}[thm]{Proposition}
\newtheorem{cor}[thm]{Corollary}

\theoremstyle{definition}

\newtheorem{defn}[thm]{Definition}

\newtheorem{remark}[thm]{Remark}

\date{}

\DeclareFontFamily{U}{mathx}{\hyphenchar\font45}
\DeclareFontShape{U}{mathx}{m}{n}{
      <5> <6> <7> <8> <9> <10>
      <10.95> <12> <14.4> <17.28> <20.74> <24.88>
      mathx10
      }{}
\DeclareSymbolFont{mathx}{U}{mathx}{m}{n}
\DeclareFontSubstitution{U}{mathx}{m}{n}
\DeclareMathAccent{\widecheck}{0}{mathx}{"71}
\DeclareMathAccent{\wideparen}{0}{mathx}{"75}

\newcommand{\e}{\varepsilon}
\newcommand\w{\omega}
\newcommand\Om{\Omega}

\newcommand\del{\partial}
\newcommand\adel{\ol{\partial}}
\newcommand\DEL{\Delta}

\newcommand\bC{{\mathbb C}}
\newcommand\bN{{\mathbb N}}

\newcommand\bR{{\mathbb R}}

\newcommand\F{{\mathcal F}}

\newcommand\M{{\mathcal M}}

\renewcommand{\O}{\mathcal{O}}

\newcommand\fg{{\mathfrak g}}

\newcommand\can{\mathrm{can}}

\newcommand\co{\mathrm{co}}

\newcommand\exd{\mathrm{d}}

\newcommand\haar{\mathrm{\bf h}}

\newcommand\unit{\mathrm{U}}
\newcommand\counit{\mathrm{C}}

\newcommand\id{\mathrm{id}}
\newcommand\proj{\mathrm{proj}}

\newcommand\inv{^{-1}}

\newcommand\coby{\, \square_{H}}
\newcommand\oby{\otimes}

\newcommand\wed{\wedge}
\newcommand\sseq{\subseteq}

\def\qbinom#1#2{\ensuremath{\left[\kern-.3em\left[\genfrac{}{}{0pt}{}{#1}{#2}\right]\kern-.3em\right]_q}}

\newcommand\ol{\overline}

\newcommand\mto{\mapsto}

\newcommand{\OO}{\mathcal{O}}

\usepackage{tikz}
\usetikzlibrary{decorations.pathreplacing}

\usepackage{ytableau}

\usepackage[english]{babel}

\def\clap#1{\hbox to 0pt{\hss#1\hss}}
\def\mathllap{\mathpalette\mathllapinternal}

\def\mathllapinternal#1#2{%
  \llap{$\mathsurround=0pt#1{#2}$}}

\newsavebox\qModFirst
\newsavebox\qModSecond
\newcommand{\Mod}{\mathrm{Mod}}
\newcommand{\modz}[2]{
  \sbox{\qModFirst}{$#1$}
  \sbox{\qModSecond}{$#2$}
  \ifdim\wd\qModFirst>\wd\qModSecond
    {}^{\phantom{#1}\mathllap{#1}}_{\phantom{#1}\mathllap{#2}}%
  \else
    {}^{\phantom{#2}\mathllap{#1}}_{\phantom{#2}\mathllap{#2}}%
  \fi\mathrm{mod}_0}
\newcommand{\qMod}[4]{{%
    \sbox{\qModFirst}{$#1$}
    \sbox{\qModSecond}{$#2$}
    \ifdim\wd\qModFirst>\wd\qModSecond
      {}^{\phantom{#1}\mathllap{#1}}_{\phantom{#1}\mathllap{#2}}%
    \else
      {}^{\phantom{#2}\mathllap{#1}}_{\phantom{#2}\mathllap{#2}}%
    \fi
    \Mod^{#3}_{#4}}}
\newcommand{\lMod}[2]{\qMod{#1}{#2}{}{}}
\newcommand{\qmod}[4]{{%
      \sbox{\qModFirst}{$#1$}
      \sbox{\qModSecond}{$#2$}
      \ifdim\wd\qModFirst>\wd\qModSecond
        {}^{\phantom{#1}\mathllap{#1}}_{\phantom{#1}\mathllap{#2}}%
      \else
        {}^{\phantom{#2}\mathllap{#1}}_{\phantom{#2}\mathllap{#2}}%
      \fi
  \mathrm{mod}^{#3}_{#4}}}
\newcommand{\lmod}[2]{\qmod{#1}{#2}{}{}}

\author[F. D\'iaz Garc\'ia]{Fredy D\'iaz Garc\'ia} 
\address{Centro de Ciencias Matem\'aticas, Universidad Nacional Aut\'onoma de M\'exico \\
Antigua Carretera a P\'atzcuaro 8701\\
Col. Ex Hacienda San Jos\'e de la Huerta\\
Morelia, C.P. 58089 M\'exico
}

 \email{lenonndiaz@gmail.com}

\author[A. Krutov]{Andrey Krutov}
\address{Department of Mathematics, University of Zagreb, Bijeni\v{c}ka cesta 30, 10000 Zagreb, Croatia}
\address{Institute of Mathematics, Czech Academy of Sciences, \v{Z}itn\'a 25, 115 67 Prague, Czech Republic}
\address{Independent University of Moscow, Bolshoj Vlasyevskij Per. 11\\ 119002 Moscow, Russia}
\email{krutov@math.cas.cz}

\author[R. \'O Buachalla]{R\'eamonn \'O Buachalla}
\address{D\'epartement de Math\'ematiques, Facult\'e des sciences, Universit\'e Libre de Bruxelles, Boulevard du Triomphe, B-1050 Bruxelles, Belgium}
\address{Mathematical Institute of Charles University, Sokolovsk\'a 83, 186 75 Prague, Czech Republic} 
\email{reamonnobuachalla@gmail.com}

\author[P. Somberg]{Petr Somberg}
\address{Mathematical Institute of Charles University, Sokolovsk\'a 83, 186 75 Prague, Czech Republic} \email{somberg@karlin.mff.cuni.cz}

\author[K. R. Strung]{Karen R. Strung}
\address{Institute of Mathematics, Czech Academy of Sciences, \v{Z}itn\'a 25, 115 67 Prague, Czech Republic}
\email{strung@math.cas.cz}

\title[Holomorphic Relative Hopf Modules]{Holomorphic Relative Hopf Modules over the Irreducible Quantum Flag Manifolds}



\begin{document}

\thanks{FDG is partially funded by Conacyt (Consejo Nacional de Ciencia y Tecnolog\'ia, M\'exico). AK was supported by the QuantiXLie Centre of Excellence, a~project cofinanced by the Croatian Government and European Union through the European Regional Development Fund --- the Competitiveness and Cohesion Operational Programme (KK.01.1.1.01.0004). R\'OB acknowledges FNRS support through  a postdoctoral fellowship within the framework of the MIS Grant ``Antipode'' grant number F.4502.18. R\'OB and PS are partially supported from the Eduard \v{C}ech Institute within the framework of the grant GA\v{C}R $19-28628X$, and by the grant GA\v{C}R $306-33/1906357$. Research of KRS and AK is supported by the GA\v{C}R project 20-17488Y and \mbox{RVO: 67985840}.}

\maketitle 

\begin{abstract}
We construct covariant  $q$-deformed holomorphic structures for all finitely generated relative Hopf modules over the irreducible quantum flag manifolds endowed with their Heckenberger--Kolb calculi. In the classical limit these reduce to modules of sections of holomorphic homogeneous vector bundles over irreducible flag manifolds. For the case of simple relative Hopf modules, we show that this covariant holomorphic structure is unique. This generalises earlier work of Majid, Khalkhali, Landi, and van Suijlekom for line modules of the Podle\'s sphere, and subsequent work of Khalkhali and Moatadelro for general quantum projective space. \end{abstract}

\section{Introduction}

In this paper we consider noncommutative generalisations of homogeneous holomorphic vector bundles for the irreducible quantum flag manifolds. Ideas from classical complex geometry have, to a greater or lesser extent, always played a role in noncommutative geometry. This is not surprising, given that no examples can claim to be more influential than the noncommutative torus $\mathbb{T}^2_{\theta}$ and the Podle\'s sphere $\O_q(S^2)$. Both are noncommutative deformations of manifolds carrying a~complex geometry in the classical limit. In fact, both $\mathbb{T}^2$ and $S^2$ are K\"ahler manifolds, the former being a Fano manifold and the latter a Calabi--Yau manifold.  Much of the classical complex geometry of these examples survives deformation intact. Of particular relevance to this paper is the work of Polishchuk and Schwarz on $\theta$-deformed holomorphic vector bundles over the noncommutative torus~\cite{PolishSch,PolishCh}, and Majid's description of the noncommutative complex geometry of the Podle\'s sphere in~\cite{Maj}. 

The notion of a {noncommutative complex structure} was introduced in~\cite{KLvSPodles} and~\cite{BS} as an abstract framework in which to study the noncommutative complex geometry of both the noncommutative torus and the Podle\'s sphere. In particular,  {holomorphic modules}, often called noncommutative holomorphic vector bundles, were introduced. The definition of a holomorphic module was motivated by the classical  Koszul--Malgrange equivalence between holomorphic vector bundles and smooth vector bundles endowed with a flat {$(0,1)$-connection}~\cite{KoszulMalgrange}. Its prototypical examples were those in~\cite{PolishCh,PolishSch} and~\cite{Maj}, as mentioned above. In particular, homogeneous line modules over the Podle\'s sphere  have canonical holomorphic structures in this sense. Moreover, they are known to satisfy a direct $q$-deformation of the classical Borel--Weil theorem~\cite{Maj,KLvSPodles}. An extension of these results to the case of general projective space was established by Khalkhali and Moatadelro~\cite{KKCPN}  using a $q$-deformed Dolbeault anti-holomorphic complex originally introduced in~\cite{SISSACPn} to construct spectral triples.


In~\cite{BeggsMajidChern}, Beggs and Majid later introduced the notion of an Hermitian holomorphic module over an algebra (where they were called Hermitian holomorphic vector bundles) and showed that any such module admits a Chern connection~\cite[Theorem 8.53]{BeggsMajidChern}, see also~\cite[\textsection 4]{BeggsMajid:Leabh}. Around the same time, noncommutative K\"ahler structures were introduced by the third author as a framework for studying noncommutative K\"ahler geometry on quantum homogeneous spaces. In joint work of the third author with \v{S}t\!'ov\'{i}\v{c}ek and van Roosmalen~\cite{OSV} it was shown that Hermitian holomorphic modules, defined over an algebra endowed with a noncommutative K\"ahler structure, have a remarkably rich theory extending the classical situation of Hermitian holomorphic vector bundles over K\"ahler manifolds. For example, the definition of a positive line bundle carries over directly, implying a natural definition of noncommutative Fano structure~\cite{OSV}. In this setting, noncommutative generalisations of  twisted Hodge theory, the Kodaira vanishing theorem, and Serre duality  all hold~\cite{OSV}, giving  powerful results and techniques with which to study holomorphic modules and their associated Dolbeault cohomology.


The next step is to enlarge the family of examples beyond the rather specialised situation of line modules over quantum projective space. A natural choice is the quantum flag manifolds $\O_q(G/L_S)$, the quantum counterpart of classical flag manifolds $G/L_S$, where $G$ is a compact Lie group and $L_S$ is a subgroup of~$G$ indexed by some subset $S$ of the simple roots of~$G$. The quantum flag manifolds form a far more general class of quantum homogeneous spaces whose theory is deeply rooted in the representation theory of Drinfeld--Jimbo quantum groups.  For the irreducible case, which is to say, those quantum flag manifolds which are irreducible (or equivalently Hermitian symmetric spaces) in the classical limit, Heckenberger and Kolb established that each quantum space has an essentially unique covariant $q$-deformed de Rham complex $\Omega^{\bullet}_q(G/L_S)$~\cite{HK,HKdR}. These remarkable differential calculi are some of the most important objects in the study of the noncommutative geometry of quantum groups.  As shown in~\cite{MMF3,MarcoConj} each $\Omega^{\bullet}_q(G/L_S)$ comes endowed with a unique covariant noncommutative K\"ahler structure. For the special case of quantum projective space, the $q$-deformed anti-holomorphic Dolbeault  complex constructed in~\cite{SISSACPn}  can be realised as a subcomplex of  the Heckenberger--Kolb calculus.


This raises the question of whether we can construct holomorphic structures for every finitely generated relative Hopf module over each irreducible quantum flag manifold. Our main result, Theorem~\ref{prop:covconnection}, shows that this is indeed possible.   To establish the result, we need to establish the existence of covariant $(0,1)$-connections and verify that they are flat. In~\cite{KKCPN} the $(0,1)$-connections for higher order projective spaces $\O_q(\mathbb{CP}^n)$ were explicitly constructed and  flatness was verified by direct calculation. Extending this approach to all irreducible quantum flag manifolds would be prohibitively lengthy and tedious. Instead, we show existence and flatness by turning to the general theory of principal comodule algebras, following an approach closer to the original constructions of Brzezinski and Majid~\cite{SMTB00, Maj}. 


While the existence of holomorphic structures is highly interesting in its own right, we are especially interested in the implications of our main result. In particular, the existence of holomorphic structures for line modules is an essential ingredient in a number of associated works. In~\cite{DOW} the holomorphic line modules over the irreducible quantum flag manifolds are shown to satisfy a direct $q$-deformation of the classical Borel--Weil theorem. This extends the case of quantum projective space discussed above, as well as the more general quantum Grassmannian picture established in~\cite{KMOS}. Building on the $q$-deformed Borel--Weil theorem,  all non-trivial line modules over the irreducible quantum flag manifolds were identified in~\cite{DOKSS} as either  positive or negative. This provides valuable information about the behaviour of their $q$-deformed Chern curvatures. These results in turn allowed Das and the third and fourth authors to establish the Fredholm property for any Dolbeault--Dirac operator twisted by a negative line module~\cite{DOSFred}. These operators then become natural candidates for spectral triples in the sense of Connes and Moscovici~\cite{ConnesMosc}. As discussed in Remark \ref{rem:BBW}, the existence of holomorphic structures also raises the question of how to explicitly describe the $U_q(\frak{g})$-module structures of the cohomology groups of the associated complexes, and how this compares with the classical situation.


The paper is organised as follows. In \textsection 2 we recall necessary preliminaries  about differential calculi, complex structures, connections, holomorphic structures, principal comodule  algebras, and strong principal connections.

In \textsection 3 we prove some general results about covariant connections for relative Hopf modules over quantum homogeneous spaces. In particular, we  show that Takeuchi's equivalence allows one to transfer flatness and uniqueness conditions for covariant connections into representation-theoretic conditions. We then observe that for a quantum homogeneous space $B = A^{\co(H)}$, cosemisimplicity of $H$ implies the existence of a left $A$-covariant strong principal connection. Using this result, we are able to produce covariant connections for finitely generated relative Hopf modules with respect to any covariant calculus.

In \textsection 4 the basic definitions and results of Drinfeld--Jimbo quantum groups and their quantised coordinate algebras are recalled. We then present the definition of the quantum flag manifolds, focusing on the special case of irreducible flag manifolds and their  Heckenberger--Kolb calculi.  We apply the general results of~\textsection 3 to the finitely generated relative Hopf modules over $\O_q(G/L_S)$, and using a representation theoretic argument, prove flatness of the $(0,1)$-connections, as well as uniqueness in the irreducible case.

We would like to thank Edwin Beggs and Adam-Christiaan van Roosmalen for useful discussions, as well as both referees for carefully reading the paper and offering a number of suggestions. In particular, we are grateful to the referee for the suggestion of Proposition~\ref{prop:torsion}.

\section{Preliminaries}

In this section we recall the necessary preliminaries  on differential calculi, complex structures, connections, holomorphic structures, and strong principal connections. Note that all algebras are defined over $\mathbb{C}$ and assumed to be unital, and all algebra maps are assumed to be unital.

\subsection{Calculi, Connections, and Holomorphic Modules}

\subsubsection{Differential Calculi}
A {\em differential calculus} $\big(\Om^\bullet \simeq \bigoplus_{k \in \bN_0} \Om^k, \exd\big)$ is a differential graded algebra (dg-algebra)  which is generated in degree $0$ as a dg-algebra, that is to say, it is generated as an algebra by the elements $a, \exd b$, for $a,b \in \Om^0$.  A {\em differential $*$-calculus} is a differential calculus equipped with a  conjugate linear involutive map  $*:\Om^\bullet \to \Om^\bullet$ satisfying $\exd(\w^*) = (\exd \w)^*$, and 
\begin{align*}
\big(\w \wed \nu\big)^*  =  (-1)^{kl} \nu^* \wed \w^*, &  & \text{ for all } \w \in \Om^k, \, \nu \in \Om^l. 
\end{align*}
For a given algebra $B$, a {\em differential calculus  over} $B$ is a differential calculus $(\Omega^\bullet, \exd)$ such that $\Om^0 = B$. Note that if $(\Omega^\bullet, \exd)$ is a differential $*$-calculus over $B$, then $B$ is a \mbox{$*$-algebra}. We say that $\omega \in \Omega^{\bullet}$ is \emph{closed} if $\exd \omega = 0$. 
See~\cite[\textsection 1]{BeggsMajid:Leabh} for a more detailed discussion of differential calculi.

\subsubsection{First-Order Differential Calculi}

A {\em first-order differential calculus} over an algebra $B$ is a pair $(\Om^1,\exd)$, where $\Omega^1$ is a $B$-bimodule and $\exd: B \to \Omega^1$ is a linear map for which the {\em Leibniz rule} holds
\begin{align*}
\exd(ab)=a(\exd b)+(\exd a)b,&  & a,b \in B,
\end{align*}
and for which $\Om^1$ is generated as a left $B$-module by those elements of the form~$\exd b$, for~$b \in B$. The {\em universal first-order differential calculus} over $B$ is the pair
$(\Om^1_u(B), \exd_u)$, where $\Om^1_u(B)$ is the kernel of the multiplication map $m_B: B \otimes B \to B$ endowed
with the obvious bimodule structure, and $\exd_u$ is the map defined by
\begin{align*}
\exd_u: B \to \Omega^1_u(B), & & b \mto 1 \otimes b - b \otimes 1.
\end{align*}
Every first-order differential calculus over $B$ is of the form $\left(\Omega^1_u(B)/N, \,\proj \circ \exd_u\right)$, where $N$ is a $B$-sub-bimodule of $\Omega^1_u(B)$ and  
$$
\proj:\Omega^1_u(B) \to \Omega^1_u(B)/N
$$
is the canonical quotient map. This gives a bijective correspondence between first-order differential calculi and sub-bimodules of $\Omega^1_u(B)$.  

We say that a differential calculus $(\Gamma^\bullet,\exd_{\Gamma})$ {\em extends} a first-order differential calculus $(\Omega^1,\exd_{\Omega})$ if there exists a bimodule isomorphism $\phi:\Omega^1 \to \Gamma^1$ such that  $\exd_{\Gamma} = \phi \circ \exd_{\Omega}$. It can be shown  that any first-order differential calculus admits an extension $\Omega^\bullet$ which is maximal  in the sense that there exists a unique dg-algebra morphism from $\Omega^\bullet$ onto any other extension of $\Omega^1$, see~\cite[\textsection 1.5]{BeggsMajid:Leabh} for details. We call this extension the {\em maximal prolongation} of the first-order differential calculus.

\subsubsection{Connections} \label{subsection:Connections}

Motivated by the Serre--Swan theorem, we think of a finitely generated projective left $B$-module $\F$ as a noncommutative generalisation of a~vector bundle.  
For $\Omega^\bullet$ a differential calculus over an algebra $B$ and $\mathcal{F}$ a finitely generated projective left $B$-module, a \emph{connection} on $\F$ is a $\mathbb{C}$-linear map 
\[
\nabla:\mathcal{F} \to \Omega^1 \otimes_B \F
\]
 satisfying 
\begin{align} \label{eqn:connLR}
\nabla(bf) = \exd b \otimes f + b \nabla f, & & \textrm{ for all } b \in B, f \in \F.
\end{align}
An immediate but important consequence of the definition is that the difference of two connections $\nabla - \nabla'$ is a left $B$-module map.

Any connection can be extended to a map $\nabla: \Omega^\bullet \otimes_B \mathcal{F} \to   \Omega^\bullet \otimes_B \mathcal{F}$ uniquely defined by 
\begin{align*}
\nabla(\omega \otimes f) =   \exd \omega \otimes f + (-1)^{|\omega|} \, \omega \wedge \nabla f,
\end{align*}
where $f \in \F$, and $\omega$ is a homogeneous element of $\Omega^{\bullet}$ of degree  $|\omega|$. The
\emph{curvature} of a connection is the left $B$-module map $\nabla^2: \mathcal{F} \to \Omega^2 \otimes_B
\mathcal{F}$. A connection is said to be {\em flat} if $\nabla^2 = 0$. Since $\nabla^2(\omega \otimes f) =
\omega \wedge \nabla^2(f)$, a connection is flat if and only if  the pair $(\Omega^\bullet \otimes_B \F,
\nabla)$ is a cochain complex.

As in the classical case, for any connection on the space of \mbox{$1$-forms} of a differential calculus we have an associated notion of torsion. (See for example Definition \cite[Definition 3.2.3]{BeggsMajid:Leabh}.)
\begin{defn} \label{defn:Torsion}
Let $(\Omega^{\bullet},\exd)$ be a differential calculus (where  $\Omega^1$ is assumed to be finitely generated and projective as a left $B$-module). The \emph{torsion} of a  connection $\nabla: \Omega^1 \to \Omega^1 \otimes \Omega^1$ is the linear operator
\begin{align*}
T_{\nabla} :=  \wedge \circ \nabla - \exd: \Omega^1 \to \Omega^1 \otimes \Omega^1.
\end{align*}
If $T_{\nabla} = 0$, then we say that $\nabla$ is \emph{torsion-free}.
\end{defn}

\subsubsection{Complex Structures}

In this subsection we recall the definition of a complex structure for a differential calculus, as introduced in~\cite{KLvSPodles, BS}, see also~\cite{BeggsMajid:Leabh}. This gives an abstract characterisation of the properties of the de Rham complex of a classical complex manifold~\cite{HUY}. 

\begin{defn}\label{defnnccs}
A {\em  complex structure} $\Om^{(\bullet,\bullet)}$, for a  differential $*$-calculus  $(\Om^{\bullet},\exd)$, is an $\bN^2_0$-algebra grading $\bigoplus_{(a,b)\in \bN^2_0} \Om^{(a,b)}$ for $\Om^{\bullet}$ such that, for all $(a,b) \in \bN^2_0$: 
\begin{enumerate}
\item \label{compt-grading}  $\Om^k = \bigoplus_{a+b = k} \Om^{(a,b)}$,
\item  \label{star-cond} $\big(\Om^{(a,b)}\big)^* = \Om^{(b,a)}$,
\item  \label{eqn:integrable} $\exd \Om^{(a,b)} \sseq \Om^{(a+1,b)} \oplus \Om^{(a,b+1)}$.
\end{enumerate}
\end{defn}

An element of $\Om^{(a,b)}$ is called an \emph{$(a,b)$-form}. For $ \proj_{\Om^{(a+1,b)}}$, and $ \proj_{\Om^{(a,b+1)}}$,  the projections from $\Om^{a+b+1}$ to $\Om^{(a+1,b)}$, and $\Om^{(a,b+1)}$ respectively, we write
\begin{align*}
\del|_{\Om^{(a,b)}} : = \proj_{\Om^{(a+1,b)}} \circ \exd, & & \ol{\del}|_{\Om^{(a,b)}} : = \proj_{\Om^{(a,b+1)}} \circ \exd.
\end{align*}
It follows from Definition~\ref{defnnccs}.3 that for any complex structure, 
\begin{align*}
\exd = \del + \adel, & &  \adel \circ \del = - \, \del \circ \adel, & & \del^2 = \adel^2 = 0. 
\end{align*}
Thus $\big(\bigoplus_{(a,b)\in \bN^2_0}\Om^{(a,b)}, \del,\ol{\del}\big)$ is a double complex. Both $\del$ and $\adel$ satisfy the graded Leibniz rule. Moreover,   
\begin{align} \label{eqn:stardel}
\del(\w^*) = \big(\adel \w\big)^*, & &  \adel(\w^*) = \big(\del \w\big)^*, & & \text{ for all  } \w \in \Om^\bullet. 
\end{align}
Associated with any complex structure $\Omega^{(\bullet,\bullet)}$ we have  a second complex structure, called its {\em opposite complex structure}, defined as
\begin{align*}
\ol{\Om}^{(\bullet,\bullet)} := \bigoplus_{(a,b) \in \bN^2_0} \ol{\Om}^{(a,b)}, & &  \textrm{ where ~~}
\ol{\Om}^{(a,b)}:= \Omega^{^{(b,a)}}.
\end{align*}
See~\cite[\textsection 1]{BeggsMajid:Leabh} or~\cite{MMF2} for a more detailed discussion of complex structures.

\subsubsection{Holomorphic Modules}

In this subsection we present the notion a holomorphic left $B$-module for an algebra $B$. Such a module should be thought of as a noncommutative holomorphic vector bundle, as has been considered in a number of previous papers, see for example~\cite{BS}, \cite{PolishSch}, and~\cite{KLvSPodles}. Indeed, the definition for holomorphic modules is motivated by the classical Koszul--Malgrange characterisation of holomorphic bundles~\cite{KoszulMalgrange}.  See~\cite{OSV} for a more detailed discussion.

With respect to a choice $\Omega^{(\bullet,\bullet)}$ of complex structure on $\Omega^{\bullet}$, a \emph{$(0,1)$-connection on $\mathcal{F}$} is a connection with respect to the differential calculus $(\Omega^{(0,\bullet)},\adel)$.

\begin{defn}
Let  $(\Omega^\bullet, \exd)$ be a differential $*$-calculus over a $*$-algebra $B$, equipped with a complex structure $\Omega^{(\bullet, \bullet)}$.  A \emph{holomorphic} left $B$-module  is a pair $(\mathcal{F},\adel_{\mathcal{F}})$, where
$\mathcal{F}$ is a finitely generated projective left $B$-module, and  $\adel_{\mathcal{F}}: \mathcal{F} \to
\Omega^{(0,1)} \otimes_B \mathcal{F}$ is a flat $(0,1)$-connection. We call $\adel_{\F}$ the \emph{holomorphic
  structure} of the holomorphic left $B$-module. 
\end{defn}

In the classical setting the kernel of the holomorphic structure map coincides with the space of holomorphic sections of a holomorphic vector bundle. This motivates us to call 
$$
H^0_{\adel}(\F) = \ker\left(\adel_{\F}: \F \to \Omega^{(0,1)} \otimes_B \F\right)\!,
$$
 the \emph{space of holomorphic sections} of $(\F,\adel_{\F})$.


\subsection{Quantum Homogeneous Spaces and Holomorphic Relative Hopf Modules}

From this point in the paper  $A$ and $H$ will \emph{always} denote Hopf algebras defined over $\mathbb{C}$, with coproduct, counit, and antipode denoted by $\Delta,\epsilon$, and $S$ respectively, without explicit reference to the Hopf algebra in question. Moreover, all antipodes are assumed to be invertible. Hence we always have an equivalence between the categories of right and left comodules of any Hopf algebra.

\subsubsection{Comodule Algebras and Quantum Homogeneous Spaces} \label{subsection:PCA}

For $H$ a Hopf algebra, and $V$ a right $H$-comodule with coaction $\DEL_R$, we say that an element $v \in V$ is {\em coinvariant} if $\DEL_R(v) = v \otimes 1$. We denote the subspace of all coinvariant elements by $V^{\co(H)}$ and call it the {\em coinvariant subspace} of the coaction $\Delta_R$.

A  {\em right  $H$-comodule algebra} $P$ is a right $H$-comodule which is also an algebra, such that the comodule
structure map $\DEL_R :P \to P \otimes H$ is an algebra map. Equivalently, it is a monoid object in $\Mod^H$,
the category of right $H$-comodules. Note that for a right $H$-comodule algebra $P$, its coinvariant subspace
$B := P^{\co(H)}$ is a subalgebra of $P$. In what follows we will always use $B$ in this sense.

If the functor $P \otimes_B -: \lMod{}{B} \to \lMod{}{\mathbb{C}}$, from the category of left $B$-modules to the category of complex vector spaces, preserves and reflects exact  sequences, then we say that $P$ is  {\em faithfully flat} as a right module over $B$. Faithful flatness for~$P$ as a  left $B$-module is defined analogously.

In this paper we are interested in a particular type of comodule algebra. Let $\pi : A \to H$ be a surjective Hopf algebra map between Hopf algebras $A$ and $H$. Then a \emph{homogeneous right $H$-coaction} is given by the map 
\begin{align*}
\Delta_R := (\id \otimes \pi) \circ \Delta : A \to A \otimes H.
\end{align*}
Note that $\Delta_R$ gives $A$ the structure of a right $H$-comodule algebra.  The associated \emph{quantum homogeneous space} is defined to be the space of coinvariant elements~$A^{\co(H)}$.

\subsubsection{Takeuchi's Equivalence} Let $B := A^{\co(H)}$ be a quantum homogeneous space. We denote by~$\lmod{A}{B}$ the category of \emph{finitely generated relative Hopf modules}, that is, the category whose  objects are left \mbox{$A$-comodules} \mbox{$\DEL_L:\mathcal{F} \to A \otimes \mathcal{F}$},  endowed with a finitely generated left $B$-module structure, such that 
\begin{align} \label{eqn:TakCompt}
\DEL_L(bf) = \Delta_L(b)\DEL_L(f),  & & \text{ for all } f \in \mathcal{F}, b \in B,
\end{align}
and whose morphism sets ${}^A_B\textrm{Hom}(-,-)$ consist of left $A$-comodule, left $B$-module, maps. It is important to note that $B$ is naturally an object in $\lmod{A}{B}$.

We denote by $\lmod{H}{}$  the category whose objects are finite-dimensional left \mbox{$H$-comodules}, and whose morphisms are left $H$-comodule maps. 
For a quantum homogeneous space $B := A^{\co(H)}$, we denote $B^+ := B \cap \ker(\epsilon)$. Consider the functor
\begin{align*}
\Phi: \lmod{A}{B} \to \lmod{H}{}, & & \mathcal{F} \mapsto \mathcal{F}/B^+\mathcal{F},
\end{align*}
where the left $H$-comodule structure of $\Phi(\mathcal{\F})$ is given by $\Delta_L[f] := \pi(f_{(-1)})
\otimes [f_{(0)}]$ (with square brackets denoting the  coset of an element  in $\Phi(\mathcal{\F})$). If $V\in \lmod{H}{}$ with coaction $\Delta_L : V \to H \otimes V$, then the {\em cotensor product} of $A$ and $V$ is given by
\begin{align*}
A \coby V := \ker(\DEL_R \oby \id - \id \oby \DEL_L: A\oby V \to A \oby H \oby V).
\end{align*}
Using the cotensor product we can define the functor 
\begin{align*}
\Psi:  \lmod{H}{} \to \lmod{A}{B}, & & V \mapsto A \,\square_H V, 
\end{align*}
where the left $B$-module and left $A$-comodule structures of $\Psi(V)$ are defined on the first tensor factor, and if $\gamma$ is a morphism in $\lmod{H}{}$, then $\Psi(\gamma) := \id \otimes \gamma$. The following equivalence was established in~\cite[Theorem 1]{Tak}. 

\begin{thm}[Takeuchi's Equivalence] \label{thm:TakEquiv} Let $B = A^{\co(H)}$ be a quantum homogeneous space such that $A$ is faithfully flat as a right $B$-module. An adjoint equivalence of categories between~$\lmod{A}{B}$ and~$\lmod{H}{}$  is given by the functors $\Phi$ and $\Psi$ and unit, and counit,  natural isomorphisms
\begin{align*}
\unit: \F \to \Psi \circ \Phi(\F), & & f \mto f_{(-1)} \otimes [f_{(0)}], \\
\counit:\Phi \circ  \Psi(V) \to V, \,  & & \Big[\sum_i a^i \otimes v^i\Big] \mto \sum_{i} \e(a^i)v^i.\end{align*} 
\end{thm}

The usual tensor product of comodules gives $\lmod{H}{}$  the structure of a monoidal category. Every object $\F \in \lmod{A}{B}$ admits a right $B$-module structure uniquely defined by 
\begin{align*}
\F \times B \to \F, & & (f,b) \mapsto f_{(-2)} b S(f_{(-1)}) f_{(0)},
\end{align*}
giving $\F$ the structure of a bimodule. The usual tensor product of bimodules then endows~$\lmod{A}{B}$ with the structure of a monoidal category. It forms a monoidal subcategory of the category of $B$-bimodules, which for sake of clarity we denote by~$\modz{A}{B}$.
Takeuchi's equivalence  can now be given the structure of a monodial equivalence in the obvious way. In particular, this means that for any monoid object~$\M \in \modz{A}{B}$ 
the corresponding $\Phi(\M) \in \lmod{H}{}$ also has the structure of a monoid object. We will use this fact tacitly throughout the paper.

\subsubsection{Relative Hopf Modules and  Covariant Connections}

 Let $\pi : A \to H$ be a surjective Hopf map and $B = A^{\co(H)}$ a quantum homogeneous space.  A differential calculus $(\Omega^\bullet, \exd)$ over $B$ is said to be \emph{covariant} if the coaction $\Delta_L : B \to A \otimes B$ extends to a (necessarily unique) map $\Delta_L : \Omega^\bullet \to A \otimes \Omega^\bullet$ giving $\Omega^\bullet$ the structure of a monoid object in~$\lmod{A}{B}$, and such that~$\exd$ is a left $A$-comodule map.  For any $\F \in \lmod{A}{B}$, a connection $\nabla : \F \to \Omega^1 \otimes_B \F$ is said to be \emph{covariant} if it is a left $A$-comodule map.

 We say that a first-order differential calculus $\Omega^1(B)$ over $B$ is   {\em left covariant} if there exists a (necessarily unique) left $A$-coaction $\DEL_L: \Omega^1(B) \to A \otimes \Omega^1(B)$ giving $\Omega^1(B)$ the structure of an object in $\lmod{A}{B}$ and such that $\exd$ is a left $A$-comodule map.  Note the universal calculus over $B$ is left $A$-covariant. Moreover, any other first-order differential calculus over $B$, with corresponding $B$-sub-bimodule $N \sseq \Omega^1_u(B)$, is covariant if and only if $N$ is a left $A$-sub-comodule of $\Omega^1_u(B)$. In particular, we note that the maximal prolongation of a covariant first-order differential calculus is covariant.

A complex structure $\Omega^{(\bullet, \bullet)}$ for $\Omega^\bullet$ is said to be \emph{covariant} if the
\mbox{$\mathbb{N}^2_0$-decomposition} of~$\Omega^{\bullet}$ is a decomposition in the category~$\lmod{A}{B}$,
or explicitly if the homogeneous subspace~$\Omega^{(a,b)}$ is a left $A$-sub-comodule of $\Omega^\bullet$, for each $(a, b) \in \mathbb{N}^2_0$. (Note that the grading implies~$\Omega^{(a,b)}$ is automatically a $B$-sub-bimodule.) For any covariant complex structure  the differentials $\del$ and $\adel$ are left $A$-comodule maps.  

\begin{defn}
A \emph{holomorphic relative Hopf module} is a pair $(\F, \adel_\F)$, where \mbox{$\F \in \lmod{A}{B}$},   \mbox{$\adel_\F : \F \to \Omega^{(0,1)} \otimes_B \F$} is a covariant $(0,1)$-connection, and $(\F, \adel_\F)$ is a holomorphic left $B$-module.
\end{defn}

\subsection{Principal Comodule Algebras and Strong Principal  Connections} \label{subsection:SPC}

In this subsection we recall the basic theory of principal comodule algebras, structures of central importance in the paper.

\subsubsection{General Case}

We say that a right $H$-comodule algebra~$(P,\Delta_R)$ is a {\em $H$-Hopf--Galois extension of} $B := P^{\co(H)}$ 
if for  $m_P:P \otimes_B P \to P$ the multiplication of~$P$,
the map
\[
\can := (m_P \otimes \id) \circ (\id \otimes \DEL_R): P \otimes_B P \to P \otimes H,
\]
is a bijection.

\begin{defn}
A {\em principal right $H$-comodule algebra}  is a right $H$-comodule algebra $(P,\DEL_R)$  such that  $P$ is a Hopf--Galois extension of $B := P^{\co(H)}$ and $P$ is faithfully flat as a right and left $B$-module.
\end{defn}

We next recall the notion of a strong principal connection for a right $H$-comodule algebra, and its relationship with the definition of principal comodule algebras.

\begin{defn}
Let $H$ be a Hopf algebra,  $P$ a right $H$-comodule algebra, and $B := P^{\co(H)}$. A \emph{principal  connection} for $P$ is a left $P$-module right $H$-comodule projection \mbox{$\Pi:\Omega^1_u(P) \to \Omega^1_u(P)$} satisfying
$$ 
\ker(\Pi) = P\Omega^1_u(B)P.
$$
A principal connection is said to be \emph{strong} if
$$
(\id - \Pi)\exd_u P \sseq \Omega^1_u(B) P.
$$
\end{defn}

As we now recall,  the existence of a strong principal connection for a comodule algebra is equivalent to the comodule algebra being principal \cite{BRzHajComptesT},\cite{BrzBohm}, see  \cite[\textsection 3.4]{TBGS} for a detailed discussion.

\begin{thm}
A comodule algebra is principal if and only if it admits a strong principal connection.
\end{thm}

\subsubsection{The Case of Quantum Homogeneous Spaces} \label{subsection:QHSPCA}

In this subsection we restrict to the special case of a quantum homogeneous space $B = A^{\co(H)}$ associated to a Hopf algebra surjection $\pi:A \to H$. First we present a natural  construction for strong principal connections. Consider $H$ as a $H$-bicomodule in the obvious way, and consider $A$ as a $H$-bicomodule with respect to the left and right $H$-coactions $\Delta_L = (\pi \otimes \id) \circ \Delta$ and $\Delta_R = (\id \otimes \pi) \circ \Delta$. Suppose that there exists a $H$-bicomodule map $i:H \to A$ splitting $\pi$, and such that $i(1_H) = 1_A$. Then a left $A$-covariant strong principal connection  is given by
\begin{align} \label{eqn:Pii}
 \Pi := m \circ (\id \otimes \omega) \circ \overline{\mathrm{can}}: \Omega^1_u(A) \to \Omega^1_u(A),
\end{align}
where $\overline{\mathrm{can}}$ is the restriction of $\mathrm{can}$ to $\Omega^1_u(A)$, and 
\begin{align*}
\omega: H \to \Omega^1_u(A), & & h \mapsto S(i(h)_{(1)})\exd_u(i(h)_{(2)}).
\end{align*}
(See \cite[\textsection 24]{MajLeabh} for further details.) Moreover, as shown in~\cite[Proposition 4.4]{SMTB00} by Brzezinski and Majid, this gives an equivalence between left $A$-covariant strong principal connections and $H$-bicomodule splittings of  $\pi$ which send the unit of $H$ to the unit of $A$.

We are interested in strong principal connections because they allow us to construct connections for  any $\F \in \lmod{A}{B}$. Consider first the isomorphism 
\begin{align*}
j: \Om^1_u(B) \oby_B \F \simeq  \Om^1_u(B)A \, \square_H \Phi(\F), & & \omega \otimes f \mapsto \omega f_{(-1)} \otimes [f_{(0)}].
\end{align*}
We claim that a strong principal connection $\Pi$ defines a connection $\nabla$ on $\F$ by
\begin{align*}
 \nabla: \F \to  \Om_u^1(B) \oby_B \F, &   & f \mto  j^{-1} \big(\big((\id - \Pi)\exd_u f_{(-1)}\big)\oby [f_{(0)}]\big).
\end{align*}
Indeed, since $\exd_u$ and $\Pi$ are both right \mbox{$H$-comodule} maps, a right $H$-comodule map is also given by the composition $(\id - \Pi) \circ \mathrm{d}_u$. Hence 
\begin{align*}
\big((\id - \Pi)\exd_u f_{(-1)}\big) \oby [f_{(0)}] \in j\left(\Om_u^1(B) \oby_B \F\right), & & \textrm{ for all \,} f \in \mathcal{F},
\end{align*}
meaning that $\nabla$ defines a  connection. We call $\nabla$ the \emph{connection for $\mathcal{F}$ associated to} $\Pi$. Note that if we additionally assume $\Pi$ to be left $A$-covariant, then it follows that the associated connection $\nabla$ is a left $A$-comodule map.

\section{Covariant Connections and Holomorphic Structures} \label{Sec:Section3}

In this section we use Takeuchi's equivalence to convert questions about existence and uniqueness of connections into representation-theoretic statements.  We also discuss principal comodule algebras and show how cosemisimplicity of a Hopf algebra $H$ can be used to construct left $A$-covariant strong principal connections. This sets up a general framework in terms of which we prove the main results of the paper in Section~\ref{sec:irredQFM}.  Recall that $A$ and $H$  denote Hopf algebras and $B  = A^{\co(H)}$ a quantum homogeneous space.

\subsection{Quotients of Connections}

In this subsection we present some elementary technical results about producing connections for non-universal calculi from connections for universal calculi.

\begin{prop} \label{prop:zeroconn}
For an algebra $B$, let $\F$ be a non-zero finitely generated projective left $B$-module and let $\Omega^{\bullet}(B)$
be a differential calculus over $B$. Then the  zero map ${\F \to \Omega^1(B) \otimes_B \F  }$ is a connection if and only if $\Omega^{\bullet}$ is the zero calculus.
\end{prop}
\begin{proof}
If the zero map were a connection, then we would necessarily have 
\begin{align*}
\exd b \otimes_B f = 0, & & \textrm{ for all } b \in B, \, f \in \F.
\end{align*}
Since $\F$ is by assumption projective as a left $B$-module, this would imply that $\exd b = 0$, for all $b \in B$, and hence that the calculus was trivial. The converse is clear, giving us the claimed equivalence.
\end{proof}

\begin{cor} \label{cor:quotientconn}
For any proper $B$-sub-bimodule $N \sseq \Omega^1_u(B)$, let us denote 
\begin{align*}
\proj_N: \Omega^1_u(B) \to   \Omega^1_u(B)/N =: \Omega^1, & & \omega \mapsto [\omega],
\end{align*}
where  $[\omega]$ denotes the coset of $\omega$ in $\Omega^1_u(B)/N$. If $ \nabla: \F \to \Omega_u^1(B)
\otimes_B \F$ is a connection with respect to the universal calculus, then a non-zero connection with
  respect to~$\Omega^1$ is given by 
\begin{align*}
\nabla': \F \to \Omega^1 \otimes_B \F,  & & f \mapsto (\proj_N \otimes \id) \circ \nabla (f).
\end{align*}
\end{cor}
\begin{proof}
It is clear from the definition of $\nabla'$ that it is a linear map satisfying the Leibniz rule \eqref{eqn:connLR}, which is to say, it is clear that $\nabla'$ is a connection. The fact that 
it is non-zero follows from Proposition \ref{prop:zeroconn} and the assumption that $N$ is a proper $B$-sub-bimodule.
\end{proof}

\subsection{Covariant Connections and Takeuchi's Equivalence}

In this subsection we make some novel observations about the flatness and uniqueness for covariant connections on a relative Hopf module $\F \in\lmod{A}{B}$. The idea is to produce sufficient criteria in terms of the morphism sets of the category~$\lmod{A}{B}$. In practical cases, this allows these questions to be transferred to representation-theoretic form, allowing for a solution by direct calculation. We first give a criteria for flatness.  

\begin{prop} \label{prop:flatness}
If ${}^A_B\mathrm{Hom}(\F,\Omega^{2} \otimes_B \F) = 0$, then any left $A$-covariant connection $\nabla:\F \to \Omega^1 \otimes_B \F$ is necessarily flat.  
\end{prop}
\begin{proof}
Since the curvature of any connection is a module map, the curvature of a covariant connection is a morphism. Thus if $^A_B\mathrm{Hom}(\F,\Omega^{2} \otimes_B \F)$ is trivial, $\nabla$ must be flat.
\end{proof}

The second proposition gives an analogous criteria for uniqueness of a covariant connection on a finitely generated relative Hopf module.

\begin{prop} \label{prop:uniqueness}
For $\F \in \lmod{A}{B}$ such that ${}^A_B\mathrm{Hom}(\F,\Omega^{1} \otimes_B \F) = 0$, there  exists at most one covariant connection for $\F$. 
\end{prop}
\begin{proof}
Since the difference of any two connections is a module map and the difference of two comodule maps is again a
comodule map, the difference of two covariant connections is a morphism in $\lmod{A}{B}$. Thus if $ {}^A_B\mathrm{Hom}(\F,\Omega^{1} \otimes_B \F)$ is trivial, then there exists at most one covariant connection $\F \to \Omega^1 \otimes_B \F$.
\end{proof}

We direct the interested reader to~\cite[\textsection 3.1]{DOKSS} for a specialisation of these results to the case of factorisable irreducible CQH-Hermitian spaces, a general framework axiomatising properties of the irreducible quantum flag manifolds and their Heckenberger--Kolb calculi as presented in  \textsection \ref{sec:irredQFM}.

\subsection{Principal Comodule Algebras and Cosemisimple Hopf Algebras}

In this subsection we discuss comodule algebras $B = A^{\mathrm{co}(H)}$ for which $H$ is a cosemisimple Hopf algebra. We begin by recalling the definition of cosemisimplicity. 

\begin{defn} \label{defn:CSS}
A Hopf algebra $A$ is  \emph{cosemisimple} if it satisfies the following
three equivalent conditions:
\begin{enumerate}
\item $A$ is the direct sum of its cosimple subcoalgebras, that is subcoalgebras which have no proper subcoalgebras,
\item  in the abelian category~$\lMod{A}{}$  of  left $A$-comodules all short exact sequences split,
\item there exists a unique linear map $\mathbf{h} : A \to \mathbb{C}$, which we call the \emph{Haar functional},
satisfying $\mathbf{h}(1) = 1$, and
\[ (\id \otimes \mathbf{h}) \circ \Delta(a) = \mathbf{h}(a) 1, \qquad (\mathbf{h} \otimes \id) \circ \Delta(a) = \mathbf{h}(a) 1.\]
\end{enumerate}
\end{defn}
For details about the equivalence of these three properties see~\cite[\S11.2]{KSLeabh}. Here we need only recall the implication from (1) to (3): With respect to the decomposition of $A$ into its simple subcoalgebras, the associated Haar functional is given by projection onto the trivial sub-coalgebra $\mathbb{C}1_A$.

Consider $\qMod{H}{}{H}{}$ the category whose objects are finite-dimensional $H$-bicomodules and whose morphisms are $H$-bicomodule maps.  In this paper, all Hopf algebras are assumed to have invertible antipodes, so we have an equivalence between $\lMod{H}{}$ the category of finite-dimensional  right $H$-comodules,  and $\Mod^H$ the category of finite-dimensional  left $A$-comodules. Hence we have an equivalence of categories 
\begin{align*}
\qMod{H}{}{H}{} \simeq \lMod{H\otimes H}{},
\end{align*}
where $H \otimes H$ is the usual tensor product of Hopf algebras. Denoting the Haar of $H$ by $\haar$,  the linear map  defined on simple tensors by
\begin{align*}
\haar_{H \otimes H} : H \otimes H \to \mathbb{C}, & & g \otimes g' \mapsto \haar(g)\haar(g'),
\end{align*}
is readily seen to be a Haar functional for $H \otimes H$ in the sense of Definition \ref{defn:CSS}. It follows that $H \otimes H$ is a cosemisimple Hopf algebra. Hence $\lMod{H \otimes H}{}$ is a semisimple abelian category, meaning that $\qMod{H}{}{H}{}$ is a semisimple abelian category. 

Let $\pi : A \to H$ be a Hopf algebra surjection and $B = A^{\co(H)}$ the associated quantum homogeneous space. It is well known that cosemisimplicity of $H$ implies that $(A,\Delta_R)$ is a 
principal $H$-comodule algebra. For example, it was shown in~\cite[Corollary 1.5]{MulSch} that 
cosemisimplicity of $H$ implies that $A$ is faithfully flat as a left and right $B$-module, 
and it follows from~\cite[Corollary 2.6]{SSHopfGalois} that $A$ is a $H$-Hopf--Galois extension of $B$.
In fact cosemisimplicity implies a  stronger result, namely the existence of a strong principal connection which is left $A$-covariant. This easy observation is undoubtedly well known to the experts, but we include a proof for sake of completeness. 

\begin{lem} \label{prop:COSS.PCA} Let $\pi\colon A \to H$ be a Hopf algebra surjection and let $\Delta_R$ denote the associated homogeneous right $H$-coaction on $A$.  If $H$ is a cosemisimple Hopf algebra, then~$\Omega^1_u(A)$ admits a left $A$-covariant strong principal  connection. In particular, $(A, \Delta_R)$ is a principal comodule algebra.
\end{lem}
\begin{proof}
Since $\pi\colon A \to H$ is a Hopf algebra map, it is necessarily a $H$-bicomodule map.  Since $H$ is cosemisimple, $^{H \otimes H}\mathrm{Mod}$ is semisimple. Hence we can choose a $H$-bicomodule map  $i\colon H \to  A$ splitting $\pi$ and satisfying $i(1_H) = 1_A$. It now follows from the discussions of \textsection \ref{subsection:QHSPCA} that $(A,\Delta_R)$ is a principal comodule algebra admitting a left $A$-covariant strong principal  connection.
\end{proof}

\begin{prop}
Assume that $H$ is cosemisimple. Let  $\F \in\lmod{A}{B}$ and  let $\Omega^{\bullet}$ be a left \mbox{$A$-covariant} differential calculus over $B = A^{\co(H)}$.
\begin{enumerate}
\item  There exists an associated left $A$-covariant connection $\nabla\colon\F \to \Omega^1 \otimes_B \F$. 

\item If we additionally assume that $\Omega^{\bullet}$ is a differential $*$-calculus endowed with a covariant complex structure~$\Omega^{(\bullet,\bullet)}$, then there will exist a left $A$-covariant $(0,1)$-connection $\adel_{\F}: \F \to \Omega^{(0,1)} \otimes_B \F$. 

\end{enumerate}
\end{prop}
\begin{proof}
Since $A$ is a principal comodule algebra, we know from the discussions in~\textsection \ref{subsection:SPC}
that $\F$ admits an associated universal connection $\nabla:\F \to \Omega^1_u(B) \otimes_B \F$. By Corollary  \ref{cor:quotientconn}, composing $\nabla$ with the quotient map $\Omega^1_u(B) \otimes_B \F \to \Omega^1 \otimes_B \F$ will give a left $A$-covariant connection
$
\nabla: \F \to \Omega^{\bullet} \otimes_B \F
$ 
for the non-universal calculus~$\Omega^{\bullet}$. Finally we note that if $\Omega^{\bullet}$ is endowed with a covariant complex structure $\Omega^{(\bullet,\bullet)}$, then 
$$
\adel_{\F}:= (\proj_{\Omega^{(0,1)}} \otimes \id) \circ \nabla: \F \to  \Omega^{(0,1)} \otimes_B \F
$$
is a left $A$-covariant $(0,1)$-connection.
\end{proof}

\section{Holomorphic Relative Hopf Modules over Quantum Flag Manifolds}
\label{sec:irredQFM}

In this section we present the primary results of the paper, namely the existence of covariant holomorphic structures for any relative Hopf module over the irreducible quantum flag manifolds, and uniqueness of such structures in the irreducible case. We first recall the necessary definitions and results about Drinfeld--Jimbo quantum groups, quantum flag manifolds, and the Heckenberger--Kolb  differential calculi over the irreducible quantum flag manifolds, and then follow with the  existence and uniqueness results for holomorphic relative Hopf modules.

\subsection{Drinfeld--Jimbo Quantum Groups}

Let $\frak{g}$ be a finite-dimensional complex semi\-simple Lie algebra of rank $r$. We fix a Cartan subalgebra $\frak{h}$ with corresponding root system $\Delta \sseq \frak{h}^*$, where $\frak{h}^*$ denotes the linear dual of $\frak{h}$.  
With respect to a choice of simple roots $\Pi = \{\alpha_1, \dots, \alpha_r\}$, denote by $(\cdot,\cdot)$ the symmetric bilinear form induced on $\frak{h}^*$ by the  Killing form of $\frak{g}$, normalised so that any shortest simple root $\alpha_i$ satisfies $(\alpha_i,\alpha_i) = 2$. The {\em coroot} $\alpha_i^{\vee}$ of a simple root $\alpha_i$ is defined by
\begin{align*}
\alpha_i^{\vee} := \frac{\alpha_i}{d_i} =  \frac{2\alpha_i}{(\alpha_i,\alpha_i)}, & & \text{ where } d_i := \frac{(\alpha_i,\alpha_i)}{2}.
\end{align*}
The Cartan matrix $A = (a_{ij})_{ij}$ of $\frak{g}$ is the $(r \times r)$-matrix defined by
$
a_{ij} := \big(\alpha_i^{\vee},\alpha_j\big).
$
Let $\{\varpi_1, \dots, \varpi_r\}$ denote the corresponding set of fundamental weights of~$\mathfrak{g}$, which is to say, the dual basis of the coroots.

Let  $q \in \bR$ such that  $q \notin \{ -1,0,1\}$, and denote $q_i := q^{d_i}$. The \emph{quantised universal enveloping algebra}  $U_q(\frak{g})$ is the  noncommutative associative  algebra  generated by the elements   $E_i, F_i, K_i$, and  $K^{-1}_i$, for $ i=1, \ldots, r$,  subject to the relations 
\begin{align*}
 K_iE_j =  q_i^{a_{ij}} E_j K_i, \quad  K_iF_j= q_i^{-a_{ij}} F_j K_i, \quad  K_i K_j = K_j K_i, \quad K_iK_i^{-1} = K_i^{-1}K_i = 1,\\
  E_iF_j - F_jE_i  = \delta_{ij}\frac{K_i - K\inv_{i}}{q_i-q_i^{-1}}, ~~~~~~~~~~~~~~~~~~~~~~~~~~~~~~~~~~~~~~~~~
\end{align*}
along with the \emph{quantum Serre relations}  
\begin{align*}
  \sum\nolimits_{s=0}^{1-a_{ij}} (-1)^s  \begin{bmatrix} 1 - a_{ij} \\ s \end{bmatrix}_{q_i}
   E_i^{1-a_{ij}-s} E_j E_i^s = 0,\quad \textrm{ for }  i\neq j,\\
  \sum\nolimits_{s=0}^{1-a_{ij}} (-1)^s \begin{bmatrix} 1 - a_{ij} \\ s \end{bmatrix}_{q_i}
   F_i^{1-a_{ij}-s} F_j F_i^s = 0,\quad \textrm{ for }  i\neq j;
\end{align*}
where we have used the $q$-binomial coefficients defined according to 
\begin{align*}
[n]_q! = [n]_q[n-1]_q \cdots [2]_q[1]_q, & & \textrm{ where \, } 
[m]_q  := \frac{q^m-q^{-m}}{q-q^{-1}}.
\end{align*}
A Hopf algebra structure is defined  on $U_q(\frak{g})$ by
\begin{align*}
\DEL(K_i) &= K_i \oby K_i,\quad  \DEL(E_i) = E_i \oby K_i + 1 \oby E_i, \quad \DEL(F_i) = F_i \oby 1 + K_i\inv  \oby F_i,\\
& \qquad S(E_i) =  - E_iK_i\inv,    \quad S(F_i) =  -K_iF_i, \quad  S(K_i) = K_i\inv,  \\
&\qquad \qquad \qquad \e(E_i) = \e(F_i) = 0, ~~ \e(K_i) = 1.     
\end{align*}  
A Hopf $*$-algebra structure, called the \emph{compact real form} of $U_q(\frak{g})$, is defined by
\begin{align*}
K^*_i : = K_i, & & E^*_i := K_i F_i, & &  F^*_i  :=  E_i K_i \inv. 
\end{align*} 

Let $\mathcal{P}$ be the weight lattice of~$\fg$, and $\mathcal{P}^+$ its set of dominant integral weights.  For each $\mu\in\mathcal{P}^+$ there exists an irreducible finite-dimensional $U_q(\frak{g})$-module  $V_\mu$, uniquely defined by the existence of a vector $v_{\mu}\in V_\mu$, which we call a {\em highest weight vector},  satisfying
\[
  E_i \triangleright v_\mu=0,\qquad K_i \triangleright v_\mu = q^{(\alpha_i, \mu)} v_\mu, \qquad
  \text{for all $i=1,\ldots,r$.}
\]
Moreover, $v_{\mu}$ is the unique such element up to scalar multiple. We call any finite direct sum of such
$U_q(\frak{g})$-representations a {\em type-$1$ representation}. In general, a non-zero vector $v\in V_\mu$ is called a \emph{weight vector} of weight~$\mathrm{wt}(v) \in \mathcal{P}$ if
\begin{align}\label{eq:Kweight}
K_i \triangleright v = q^{(\mathrm{wt}(v), \alpha_i)} v, & & \textrm{ for all } i=1,\ldots,r.
\end{align}
Finally, we note that since $U_q(\frak{g})$ has an invertible antipode, we have an equivalence between $\lMod{}{U_q(\frak{g})}$, the category of  left $U_q(\frak{g})$-modules, and $\Mod_{U_q(\frak{g})}$, the category of right $U_q(\frak{g})$-modules, as induced by the antipode. Explicitly, for any left $U_q(\frak{g})$-module $V$, its right $U_q(\frak{g})$-module structure is determined by 
\begin{align*}
v \triangleleft X = S^{-1}(X) \triangleright v, & & \text{ for all } v \in V, \, X \in U_q(\frak{g}).
\end{align*}
For further details on Drinfeld--Jimbo quantised enveloping algebras, we refer the reader to the standard
texts~\cite{ChariPressley, KSLeabh}, or to the seminal papers~\cite{DrinfeldICM, Jimbo1986}.

\subsection{Quantum Coordinate Algebras} 
In this subsection we recall some necessary material about quantised coordinate algebras.
Let $V$ be a finite-dimensional left $U_q(\frak{g})$-module, $v \in V$, and $f \in V^*$, the $\mathbb{C}$-linear dual of $V$, endowed with its  right \mbox{$U_q(\frak{g})$-module} structure. Let us note that, with respect to the equivalence between type-1 $U_q(\fg)$-modules and
finite-dimensional representations of~$\fg$, the left module corresponding to $V^*_{\mu}$ is isomorphic to $V_{-w_0(\mu)}$, where $w_0$ denotes the longest element in the Weyl group of $\frak{g}$.

Consider the function  $c^{\textrm{\tiny $V$}}_{f,v}:U_q(\frak{g}) \to \bC$ defined by $c^{\textrm{\tiny $V$}}_{f,v}(X) := f\big(X \triangleright v\big)$. The \emph{space of matrix coefficients} of $V$ is the subspace
\begin{align*}
C(V) := \text{span}_{\mathbb{C}}\!\left\{ c^{\textrm{\tiny $V$}}_{f,v} \,| \, v \in V, \, f \in V^*\right\} \sseq U_q(\frak{g})^*.
\end{align*}
A $U_q(\fg)$-bimodule structure on~$C(V)$ is given by
\begin{equation} \label{eq:Zact}
  (Y\triangleright c^{\textrm{\tiny $V$}}_{f,v} \triangleleft Z)(X) := f\left((ZXY)\triangleright v\right)
  = c^{\textrm{\tiny $V$}}_{f\triangleleft Z, Y \triangleright v}  (X) = c^{\textrm{\tiny $V$}}_{S^{-1}(Z)\, \triangleright f, Y \triangleright v}  (X),
\end{equation}
for all $X,Y,Z\in U_q(\frak{g})$.  Let $U_q(\frak{g})^\circ$ denote the Hopf dual of $U_q(\frak{g})$. It is easily checked that a Hopf  subalgebra of $U_q(\frak{g})^{\circ}$ is given by
\begin{equation}\label{eq:PeterWeyl}
\OO_q(G) := \bigoplus_{\mu \in \mathcal{P}^+} C(V_{\mu}).
\end{equation}
We call $\OO_q(G)$ the {\em quantum coordinate algebra of~$G$}, where~$G$ is the compact, connected, simply-connected, simple Lie group  having~$\frak{g}$ as its complexified Lie algebra.
Note that $\OO_q(G)$ is a cosemisimple Hopf algebra by construction.

\subsection{Quantum Flag Manifolds} \label{subsection:QFM}

For $\{\alpha_i\}_{i\in S} \sseq \Pi$ a subset of simple roots,  consider the Hopf $*$-subalgebra
\begin{align*}
U_q(\frak{l}_S) := \big< K_i, E_j, F_j \,|\, i = 1, \ldots, r; j \in S \big>.
\end{align*} 
Just as for $U_q(\frak{g})$, see for example~\cite[\textsection 7]{KSLeabh}, the category of $U_q(\frak{l}_S)$-modules is semisimple. The Hopf $*$-algebra embedding $\iota_S:U_q(\frak{l}_S) \hookrightarrow U_q(\frak{g})$ induces the dual Hopf \mbox{$*$-algebra} map $\iota_S^{\circ}: U_q(\frak{g})^{\circ} \to U_q(\frak{l}_S)^{\circ}$. By construction $\OO_q(G) \sseq U_q(\frak{g})^{\circ}$, so we can consider the restriction map
\begin{align*}
\pi_S:= \iota_S^{\circ}|_{\OO_q(G)}: \OO_q(G) \to U_q(\frak{l}_S)^{\circ},\
\end{align*}
and the Hopf $*$-subalgebra 
$
\OO_q(L_S) := \pi_S\big(\OO_q(G)\big) \sseq U_q(\frak{l}_S)^\circ.
$
The {\em quantum flag manifold associated} to $S$ is the quantum homogeneous space associated to the surjective Hopf $*$-algebra map  $\pi_S:\OO_q(G) \to \OO_q(L_S)$. We denote it by
\begin{align*}
\OO_q\big(G/L_S\big) := \OO_q \big(G\big)^{\text{co}\left(\OO_q(L_S)\right)}.
\end{align*} 
Since the category of $U_q(\frak{l}_S)$-modules is semisimple,  $\O_q(L_S)$ must be a cosemisimple Hopf algebra. Thus by Proposition \ref{prop:COSS.PCA}, the pair $(\O_q(G), \Delta_R)$ is a principal comodule algebra.

Denoting $\mu_S := \sum_{s\notin S} \varpi_s$, choose for $V_{\mu_S}$ a weight basis $\{v_i\}_{i}$, with corresponding dual basis $\{f_i\}_{i}$.  As shown in~\cite[Proposition 3.2]{HK}, writing $N := \dim(V_{\mu_S})$, a set of generators for $\OO_q(G/L_S)$ is given by 
\begin{align*}
z_{ij} := c^{V_{\mu_S}}_{f_i,v_N}c^{V_{-w_0(\mu_S)}}_{v_j,f_N}  & & \text{ for } i,j = 1, \dots, N,
\end{align*}
where $v_N$ is a highest weight vector of~$V_{\mu_S}$, and $f_N$ is a lowest weight vector of~$V_{-w_0(\mu_S)}$.

\subsection{The Heckenberger--Kolb Calculi and their Complex Structures}

The construction and classification of covariant differential calculi over the quantum flag manifolds poses itself as a very important and challenging question. At present this question has only been addressed for the irreducible quantum flag manifolds, a distinguished sub-family  whose definition we now recall.

\begin{defn}
A quantum flag manifold is \emph{irreducible} if the defining subset of simple roots is of the form
$$
S = \{1, \dots, r \} \setminus \{s\}
$$
where $\alpha_s$ has coefficient $1$ in the expansion of the highest root of $\frak{g}$.
\end{defn}

In the classical limit of $q=1$, these homogeneous spaces reduce to the compact Hermitian symmetric spaces, see for example~\cite[Table 10.1]{BastonEastwood} or~\cite[\textsection X.3]{Helgason1978}. For a convenient diagrammatic presentation of the explicit simple roots identified by this condition, as well as the dimensions of the classical differential manifolds, see~\cite[Appendix B]{DOKSS}.

The irreducible quantum flag manifolds are distinguished by the existence of an essentially unique $q$-deformation of their classical de Rham complexes. The existence of such a canonical deformation is one of the most important results in the noncommutative geometry of quantum groups, establishing it as a solid base from which to investigate more general classes of quantum spaces. The following theorem is a direct consequence of results established in~\cite{HK}, \cite{HKdR}, and~\cite{MarcoConj}. See~\cite[\textsection10]{DOSFred} for a more detailed presentation.

\begin{thm}\label{thm:HKClass}
Over any irreducible quantum flag manifold $\OO_q(G/L_S)$, there exists a unique finite-dimensional left $\OO_q(G)$-covariant differential $*$-calculus
\[
\Omega^{\bullet}_q(G/L_S) \in \modz{\OO_q(G)}{\OO_q(G/L_S)},
\]
 of {classical dimension}, that is to say, satisfying
  \begin{align*}
    \dim \Phi\!\left(\Omega^{k}_q(G/L_S)\right) = \binom{2M}{k}, & & \text{ for all \,} k = 0, \dots, 2 M,
  \end{align*}
  where $M$ is the complex dimension of the corresponding classical manifold.
\end{thm}

The calculus $\Omega^{\bullet}_q(G/L_S)$, which we refer to as the \emph{Heckenberger--Kolb calculus} of $\OO_q(G/L_S)$,  has many remarkable properties. We recall here only the existence of a unique covariant complex structure, following from the results of~\cite{HK}, \cite{HKdR}, and~\cite{MarcoConj}.

\begin{prop} \label{prop:complexstructure}
Let $\OO_q(G/L_S)$ be an irreducible quantum flag manifold, and $\Omega^{\bullet}_q(G/L_S)$  its Heckenberger--Kolb differential $*$-calculus. Then the following hold:
\begin{enumerate}
\item $\Omega^{\bullet}_q(G/L_S)$ admits precisely two left $\OO_q(G)$-covariant complex structures, each of which is opposite to the other, 
\item for each complex structure $\Omega^{(1,0)}$ and $\Omega^{(0,1)}$ are simple objects in $\modz{\OO_q(G)}{\OO_q(G/L_S)}$.
\end{enumerate}
\end{prop}

Complementing this abstract characterisation of the calculus is the original presentation of Heckenberger and Kolb  in terms of the generators  $z_{ij} \in \O_q(G/L_S)$ given in~\cite{HKdR}. Here we need only recall the following: Consider the subset of the index set $J:= \{1,\ldots,\dim(V_{\varpi_s})\}$ given by
$$
J_{(1)} := \{ i\in J \mid (\varpi_s,\varpi_s - \alpha_s - \mathrm{wt}(v_i)) = 0\},
$$
where $\{v_i\}_{i\in J}$ is a weight basis of~$V_{\varpi_s}$. It follows from~\cite[Proposition~3.6]{HKdR} that we can make an explicit choice of left $A$-covariant complex structure
$$
\Omega^{\bullet}_q(G/L_S) \simeq  \bigoplus_{(a,b)\in\mathbb{N}_0^2} \Omega^{(a,b)} =: \Omega^{(\bullet,\bullet)}
$$ 
uniquely defined by the fact that a basis of $\Phi(\Omega^{(0,1)})$ is given by 
\begin{align} \label{eqn:HKbasis}
\left\{[\adel z_{Ni}] \,| \, \text{ for } i\in J_{(1)}\right\}\!.
\end{align}

\subsection{Holomorphic Modules}

Here we establish existence and uniqueness results of holomorphic structures for relative Hopf modules over the irreducible quantum flag manifolds. We begin by observing that the general theory of principal comodule algebras, together with cosemisimplicity of $\O_q(L_S)$, implies the existence of covariant connections.

\begin{lem} \label{lem:Ecovconnection}
Let $\OO_q(G/L_S)$ be an irreducible quantum flag manifold  endowed with its Heckenberger--Kolb calculus $\Omega^{\bullet}_q(G/L_S)$.  Every $\F \in \modz{\OO_q(G)}{\OO_q(G/L_S)}$ admits a left $\OO_q(G)$-covariant connection $\nabla: \F \to \Omega^1_q(G/L_S) \otimes_{\OO_q(G/L_S)} \F$. 
\end{lem}
\begin{proof}
As observed in \textsection \ref{subsection:QFM}, each irreducible quantum flag manifold is a principal comodule algebra. Following the discussion of  \textsection \ref{subsection:SPC}, this implies the existence of a covariant universal connection for each $\F \in \modz{\OO_q(G)}{\OO_q(G/L_S)}$. Corollary \ref{cor:quotientconn} now implies that we can quotient this connection to produce a covariant connection with respect to the Heckenberger--Kolb calculus $\Omega^{\bullet}_q(G/L_S)$.
\end{proof}

We now use Proposition \ref{prop:uniqueness} to show uniqueness for covariant connections whenever $\F$ is simple. This is most easily done by considering $\Phi(\F)$ as a module over  the centre of $U_q(\frak{l}_S)$. Recalling that the transpose of the Cartan matrix $A$ is the change of basis matrix taking fundamental weights to simple roots, we see that $\det(A) \varpi_s$ is contained in the root lattice of $\frak{g}$. Denoting
$$
\det(A) \varpi_s =: a_1 \alpha_1 + \cdots + a_r \alpha_r,
$$
it follows directly from the commutation relations of~$U_q(\fg)$ that 
\begin{align*}
Z := K^{a_1}_1 \cdots K^{a_r}_r
\end{align*}
is a central and grouplike element of $U_q(\frak{l}_S)$. 
Recall that the elements of the centre~$\frak{z}(U_q(\frak{l}_S))$ of $U_q(\frak{l}_S)$ act on any irreducible
$U_q(\frak{l}_S)$-module~$V$ by a corresponding central
character~$\chi_V$, which is to say, an element of $\mathrm{Hom}(\frak{z}(U_q(\frak{l}_S)), \mathbb{C})$ the set of algebra maps from
$\frak{z}(U_q(\frak{l}_S))$ to $\mathbb{C}$.

Recall that $\Phi(\Omega^{(0,1)})$ is spanned by the elements of the form~\eqref{eqn:HKbasis}. Therefore, for any $i\in
J_{(1)}$, we have
\[
  [\bar\del z_{Ni}] \triangleleft Z
  = [\bar\del (c^{\varpi_s}_{f_N\triangleleft Z,v_N}c^{-w_0(\varpi_s)}_{v_i\triangleleft Z,f_N})]
  = q^{(\varpi_s,\alpha_s)\det(A)}[\bar\del z_{Ni}].
\]
Hence, $\chi_{\Phi(\Omega^{(0,1)})}(Z) = q^{(\varpi_s,\alpha_s)\det(A)}$. Thus since 
$(\varpi_s,\alpha_s) \neq 0$  we have that 
\begin{equation}\label{eq:ZOm01Char}
  \chi_{\Phi(\Omega^{(0,1)})}(Z)  \neq 1,-1.
\end{equation}
Note also that for $V$ and $W$ two irreducible $U_q(\frak{l}_S)$-modules
the grouplike elements of the centre~$\frak{z}(U_q(\frak{l}_S))$ act on $V \otimes W$ by a central character $\chi_{V\otimes W}$ according to 
\begin{align*}
\chi_{V\otimes W}(x) = \chi_V(x)\chi_W(x), & & \textrm{ for any grouplike } x \in \frak{z}(U_q(\frak{l}_S)).
\end{align*}

\begin{thm} \label{prop:covconnection}
Let $\OO_q(G/L_S)$ be an irreducible quantum flag manifold  endowed with its Heckenberger--Kolb calculus, and
$\F \in \modz{\O_q(G)}{\O_q(G/L_S)}$. It holds that 
\begin{enumerate}
\item $\F$ admits a left $\O_q(G)$-covariant connection $\nabla: \F \to \Omega^1_q(G/L_S) \otimes_{\O_q(G/L_S)} \F$, and this is the unique such connection if $\F$ is simple,
\item $\adel_{\F} := \mathrm{proj}^{(0,1)} \circ \nabla$ is a left $\O_q(G)$-covariant holomorphic structure for $\F$, and this is the unique such holomorphic structure if $\F$ is simple.
\end{enumerate}
\end{thm}
\begin{proof}
1. By Lemma \ref{lem:Ecovconnection} a covariant connection exists. Assuming that $\F$ is simple,  it follows from \eqref{eq:ZOm01Char} that 
  \[
    \chi_{\Phi(\Omega^{(0,1)}) \otimes \Phi(\F)}(Z) = \chi_{\Phi(\Omega^{(0,1)})}(Z) \chi_{\Phi(\F)}(Z) \neq \chi_{\Phi(\F)}(Z).
  \]
By Schur's lemma we conclude that there are no non-zero $U_q(\frak{l}_S)$-module maps from $\Phi(\F)$ to $\Phi(\Omega^{(0,1)}) \otimes \Phi(\F)$. Moreover, since we have a non-degenerate dual pairing between $U_q(\frak{l}_S)$ and $\mathcal{O}_q(L_S)$, there are no non-zero $\O_q(L_S)$-comodule maps from~$\Phi(\F)$ to $\Phi(\Omega^{(0,1)}) \otimes \Phi(\F)$. This in turn implies that there can exist no non-zero morphisms from $\F$ to $\Omega^{(0,1)} \otimes_{\O_q(G/L_S)} \F$. Proposition~\ref{prop:uniqueness} now implies that there exists at most one left $\OO_q(G)$-covariant connection on~$\F$ for the calculus~$\Omega^{(0,\bullet)}$. An analogous argument shows that there exists at most one left $\OO_q(G)$-covariant connection on~$\F$ for the calculus~$\Omega^{(\bullet,0)}$. Hence if~$\F$ is simple, then \mbox{$\nabla: \F \to \Omega^1_q(G/L_S) \otimes_{\O_q(G/L_S)} \F$} is the unique such covariant connection.

\bigskip

2.  Since $\Omega^{\bullet}$ is a monoid object in $\modz{\O_q(G)}{\O_q(G/L_S)}$, we see that $\Phi(\Omega^{\bullet})$ has the structure of a monoid object in $\lmod{\O_q(L_S)}{}$, or in other words,  it has the structure of a left $\O_q(L_S)$-comodule algebra. In particular, for any two forms $\omega, \nu \in \Omega^{\bullet}$, it holds that
\begin{align*}
([\omega] \wedge [\nu]) \triangleleft Z = ([\omega] \triangleleft Z) \wedge ([\nu] \triangleleft Z).
\end{align*}
Thus we see that 
\begin{align*}
  \chi_{\Phi(\Omega^{(0,2)})}(Z)  = \left(\chi_{\Phi(\Omega^{(0,1)})}(Z)\right)^2.
\end{align*}
From this we see that, for any irreducible  $\F$,
\[
  \chi_{\Phi(\Omega^{(0,2)}) \otimes \Phi(\F)}(Z)  = \left(\chi_{\Phi(\Omega^{(0,1)})}(Z) \right)^2 \chi_{\Phi(\F)}(Z) \neq \chi_{\Phi(\F)}(Z),
\]
where we have used  \eqref{eq:ZOm01Char}. Following the same argument as for $(0,1)$-forms in part 1 of the proof, this means that there are no non-zero {$\O_q(G)$-comodule} maps from~$\F$ to~$\Omega^{(0,2)} \otimes_{\O_q(G/L_S)} \F$. Flatness of the $(0,1)$-connection $\adel_{\F}$ now follows from Proposition \ref{prop:flatness}. Uniqueness was already established in 1.

For the case of a non-simple $\F$, cosemisimplicity of $\O_q(L_S)$ implies that $\F$ is a direct sum of a finite number of simple objects $\F \simeq \bigoplus_{i} \F_i$. The direct sum of the covariant holomorphic structures of the summands $\F_i$ gives a covariant holomorphic structure for $\F$.
\end{proof}

Let us now consider the torsion (see Definition \ref{defn:Torsion}) for the special case of the connection
$$
\adel_{\Omega^{(0,1)}}: \Omega^{(0,1)} \to \Omega^{(0,1)} \otimes_{\O_q(G/L_S)} \Omega^{(0,1)}.
$$
Since $T_{\adel_{\Omega^{(0,1)}}}$ is necessarily a left $\O_q(G/L_S)$-module map, we can consider its  image under the functor $\Phi$. As was argued in the proof of Theorem~\ref{prop:covconnection}, it follows from \eqref{eq:ZOm01Char} that 
$$
\Phi(T_{\adel_{\Omega^{(0,1)}}}) = 0.
$$
This gives us the following result.

\begin{prop} \label{prop:torsion}
The connection $\adel_{\Omega^{(0,1)}}$ is torsion-free.
\end{prop}


\begin{remark} \label{rem:BBW}
For each ${\mathcal{F}} \in \modz{\O_q(G)}{\O_q(G/L_S)}$, the existence of a holomorphic structure $\adel_{\mathcal{F}}$ gives a complex
$$
\adel_{{\mathcal{F}}}: \Omega^{k} \otimes_{\O_q(G/L_S)} {\mathcal{F}}  \to \Omega^{k+1} \otimes_{\O_q(G/L_S)} {\mathcal{F}},
$$
with associated cohomology groups
$$
H^k_{\adel}(\F) \simeq  \bigoplus_{(a,b) \in \mathbb{N}^2_0} H^{(a,b)}_{\adel}(\F). 
$$
By the Hodge decomposition theorem for noncommutative K\"ahler structures \cite[Theorem 6.4]{OSV}, with respect to some choice of $\O_q(G)$-covariant Hermitian structure for $\F$ (see  \cite[Definition 2.13]{OSV}),  each cohomology group $H^{(a,b)}_{\adel}(\F)$ can be identified with the space of \emph{$(a,b)$-harmonic forms}  
$$
\mathcal{H}_{\adel}^{(a,b)}(\F) := \ker\!\left(\Delta_{\adel}:\Omega^{(a,b)} \otimes_{\O_q(G/L_S)} \F \to \Omega^{(a,b)} \otimes_{\O_q(G/L_S)} \F\right)\!.
$$  
Since the associated Laplacian $\Delta_{\adel}$ is a $U_q(\frak{g})$-module map, the space $\mathcal{H}^{(a,b)}_{\adel}(\F)$ carries an action of $U_q(\frak{g})$, and hence induces a $U_q(\frak{g})$-module structure on $H^{(a,b)}_{\adel}(\F)$. Finding an explicit description of these $U_q(\frak{g})$-modules, and comparing them with their classical counterparts (see \cite{BottBW} and  \cite{BorelWeilSerre}) is an important and interesting challenge.
\end{remark}

\begin{remark}
The proof that $\adel_{\F}$ is a flat $(0,1)$-connection is readily adapted to show that   
$
\del_{\F} := \mathrm{proj}^{(1,0)} \circ \nabla
$
satisfies 
\begin{align} \label{eqn:antiholostructure}
\del_{\F}^2 = 0.
\end{align}
This can be described more formally in terms of opposite complex structures. Indeed, \eqref{eqn:antiholostructure} says that $\del$ is a holomorphic structure with respect to the opposite complex structure of $\Omega^1_q(G/L_S)$. Moreover, by the same argument given for Proposition~\ref{prop:torsion}, this holomorphic structure is torsion-free. 
\end{remark}

 \bibliographystyle{siam}

\end{document}